\documentclass[11pt]{amsart}

\usepackage{amssymb,amsmath,amsfonts,amsthm}
\usepackage{geometry}
\usepackage{latexsym}
\usepackage{graphicx}
\usepackage{indentfirst}

\newtheorem{theorem}{Theorem}
\newtheorem{lemma}[theorem]{Lemma}
\newtheorem{corollary}[theorem]{Corollary}

\newtheorem{definition}{Definition}

\begin{document}

\title{The DP-coloring of the square of subcubic graphs}

\author{Ren Zhao}
\date{\today}
\address{Department of Fundamental Education\\ Suzhou Vocational and Technical College\\ P.R. China } 
\email{ren-zhao@outlook.com}

\maketitle

\begin{abstract}
	The 2-distance coloring of a graph $G$ is equivalent to the proper coloring of its square graph $G^2$, it is a special distance labeling problem. DP-coloring (or “Correspondence coloring”) was introduced by Dvořák and Postle in 2018, to answer a conjecture of list coloring proposed by Borodin. In recent years, many researches pay attention to the DP-coloring of planar graphs with some restriction in cycles. We study the DP-coloring of the square of subcubic graphs in terms of maximum average degree $\rm{mad}(G)$, and by the discharging method, we showed that: for a subcubic graph $G$, if $\rm{mad}(G)<9/4$, then $G^2$ is DP-5-colorable; if $\rm{mad}(G)<12/5$, then $G^2$ is DP-6-colorable. And the bound in the first result is sharp.

\medskip
  


\end{abstract}

\section{Introduction}\label{intro}

For simple graph $G=(V,E)$, we denote $d_G(v),\Delta(G),\delta(G)$ to be the degree, maximum degree,minimum degree of $G$,respectively. A cubic graph is a graph with all vertices have degree $3$, a subcubic graph is a graph with $\Delta(G)\le3$. The square graph $G^2$ of $G$ is the graph obtained by joining any pair of vertices of distance 2 with a new edge. Let $\rm{mad}(G)=max_{H\subseteq G}{\frac{2|E(H)|}{|V(H)|}}$ be the maximum average degree of $G$. And the girth $g$ of the $G$ is the length of the shortest cycle in $G$. A $k$-vertex is a vertex of degree $k$, a $k$-cycle is a cycle of length $k$, a $k$-face is a face with a $k$-cycle be its boundary cycle. We use $E[X,Y]$ for the set consists of the edges of $G$ with one endpoint in $X\subseteq V$ and another in $Y\subseteq V$. 

The \textbf{proper $k$-coloring} of $G$ is a mapping $c:V\rightarrow \{1,2,\cdots,k\}$, satisfing $c(v_0)\neq c(v_1)$ for any $v_0v_1\in E$. The \textbf{chromatic number},denoted by $\chi(G)$, is the smallest $k$ such that $G$ has a proper $k$-coloring. Moreover,if $c(v_0)\neq c(v_2)$ for any  $v_2\in V$ has a common neighbor with $v_0$, we say $c$ is a \textbf{2-distance $k$-coloring} of $G$. We define $\chi_2(G)$ as the smallest $k$ such that $G$ admits a 2-distance $k$-coloring. Apparently, the proper coloring of $G^2$ is equivalent to the 2-distance coloring of $G$, that is to say:$\chi_2(G)=\chi(G^2)$.

For a list assignment $L$ defined on $V$, if there exist a proper coloring $c$, such that $c(v)\in L(v), v\in V$, we call $c$ a $L$-coloring of $G$ or $G$ is $L$-colorable. Graph $G$ is \textbf{$k$-choosable} if $G$ is $L$-colorable for any $L$ with $|L(v)|\ge k,v\in V$. The minimum $k$ such that $G$ is $k$-choosable is the \textbf{list chromatic number}, denoted by $\chi_l(G)$.

As the generalization of proper coloring, $\chi_l(G)$ is no less than $\chi(G)$. In contrast to all the planar graphs are 4-colorable(4-Color Theorem), Thomassen\cite{T94} showed that they are 5-choosable. Naturaly, there are many articles concerning the 4-choosability and 3-choosability of some planar graphs without certain cycles. In 1996, Borodin\cite{B96} proved that: every planar graph without cycles between 4 and 9 is 3-colorable. Then in 2013, Borodin\cite{B13} proposed a problem that: prove that every planar graph without cycles of length from 4 to 8 is 3-choosable[Problem 8.1]. Dvořák and Postle\cite{DP18} gave an affirmative answer by introducing a new concept:\textbf{Correspondence Coloring}. It was renamed DP-coloring by Bernshtepn, Kostochka and Pron\cite{BKP17}, and they repharsed the definition as following:
\begin{definition}
	For graph $G=(V,E)$ and the list assignment $L$, the \textbf{cover} $H_L$ of $G$ is a graph that satisfy the following conditions:
	\begin{enumerate}
		\item $V(H_L)=\{(u,c_i):u\in V,c_i\in L(u)\}$;
		\item The induced subgraph of $\{u\}\times L(u)=\{(u,c_i):c_i\in L(u)\}$ is a clique;
		\item $E[\{u\}\times L(u),\{v\}\times L(v)], uv\in E$ form a matching of $H_L$, denoted by $M_{uv}$(may be empty);
		\item $E[\{u\}\times L(u),\{v\}\times L(v)]=\emptyset, uv\notin E$.
	\end{enumerate} 
	Let $M_L=\{M_{uv}:uv\in E\}$ be the matching assignment of $H_L$. For any list assignment $L$ with $|L(u)|\ge k,u\in V$, We say $G$ is \textbf{DP-$k$-colorable} if there exist an independent set $I$ with $|I|=|V|$ in $H_L$. The \textbf{DP-chromatic number} is the minimum $k$ such that $G$ is DP-$k$-colorable, denoted by $\chi_{DP}(G)$.
\end{definition}
For graph $G$ and list assignment $L$, if $G$ is DP-$k$-colorable for matching assignment $M_L=\{(u,c_i)(v,c_i):uv\in E,c_i\in L(u)\cap L(v)\}$, then the independent set $I=\{(u,c(u)):u\in V,c(u)\in L(u)\}$ is corresponding to a $L$-coloring $c$ of $G$. That is to say, DP-coloring is a generalization of list coloring, and $\chi_{DP}(G)\ge \chi_l(G)$. For example, $\chi_{DP}(C_4)=3, \chi_l(C_4)=2$.
	
In this article, we mainly concentrate on the DP-coloring to the square of the subcubic graphs. In terms of proper coloring: Wegner\cite{W77} proved that the square of the cubic graph is 8-colorable in 1977; Thomassen\cite{T18} showed that the square of a planar cubic graph is 7-colorable. In terms of list coloring, Dvořak: For subcubic graph $G$, Skrekovski and Tancer\cite{DST08} showed that:$\chi_l(G^2)\le 4$ if $\rm{mad}(G)<\frac{24}{11}$ and without $5$-cycles, $\chi_l(G^2)\le 5$ if $\rm{mad}(G)<\frac{7}{3}$, and $\chi_l(G^2)\le 6$ if $\rm{mad}(G)<\frac{5}{2}$; Havet\cite{H09} improved the bound of the third result to $\frac{18}{7}$; Cranston and Kim\cite{CK08} showed that $\chi_l(G^2)\le 8$ if $G$ is not Petersen graph(and the bound is sharp). We showed that:
\begin{theorem}\label{thm1}
	For subcubic graph $G$,$\chi_{DP}(G^2)\le5$ if $\rm{mad}(G)<\frac{9}{4}$.
\end{theorem}
\begin{theorem}\label{thm2}
	For subcubic graph $G$,$\chi_{DP}(G^2)\le6$ if $\rm{mad}(G)<\frac{12}{5}$.
\end{theorem}

 Let us now consider the planar subcubic graphs $G$ under girth $g$ constraints. Montassier and Raspaud\cite{MR06} showed that $G^2$ is $5$-colorable for $g\ge 14$, $6$-colorable for $g\ge 10$ and $7$-colorable for $g\ge 8$;  Dvořak, Skrekovski and Tancer\cite{DST08} obtained some results about list coloring: $G^2$ is $4$-choosable for $g\ge24$, $5$-choosable for $g\ge 14$, $6$-choosable for $g\ge 10$; Cranston and Kim\cite{CK08} improved the girth bound by showing that:$G^2$ is $6$-choosable for $g\ge9$(this result was also obtained by Havet\cite{H09}), $7$-choosable for $g\ge7$. And later the girth bound of $5$-choosable was improved to $13$ by Havet\cite{H09}, and $12$ by Borodin\& Invanova\cite{BI12a}. Borodin\& Invanova\cite{BI11} also proved that $G^2$ is $4$-colorable if $g\ge23$, and they improved this result to $22$ in \cite{BI12b}, now $21$ is the best result which was given by Hoang La and Montassier\cite{LM21}.

By Euler’s Formula, there is a folklore result of planar graph $G$ with girth at least $g$ says: $\rm{mad}(G)<\frac{2g}{g-2}$. Together with above theorems, the following results can be obtained easily:

\begin{corollary}
	Let $G$ be a planar subcubic graph of girth $g$, then $\chi_{DP}(G^2)\le5$ if $g>18$.
\end{corollary}
\begin{corollary}
	Let $G$ be a planar subcubic graph of girth $g$, then $\chi_{DP}(G^2)\le6$ if $g>12$.
\end{corollary}

\section{Terminology}
A graph $G$ is \textbf{$k$-minimal} if $G^2$ is not DP-$k$-colorable but each of its proper subgraph does. So $k$-minimal graph is obviously connected. A configuration is \textbf{$k$-reducible} if it can not appear in the $k$-minimal graph $G$. A \textbf{$l$-thread} is a path induced by $l$ vertices of degree $2$ in $G$, and a $0$-thread is a $3$-vertex. A $l$-thread is called \textbf{longest} if its two endpoingts are both $3$-vertices. Let $Y_{a,b,c}$ be a $3$-vertex incident with three threads of length $a,b,c$. Let $G_2(v)$ be the subgraph induced by a vertex and the longest threads it incident with.

\section{Proofs of Results}
We proof the main results by contradiction. If Theorem~\ref{thm1} is not true, there must exist some subcubic graphs with maximum average degree less than $\frac{9}{4}$ while their square graphs are not DP-$5$-colorable. Let $G$ be the counterexample with the fewest vertices, named $5$-minimal graph. Then the square of any proper subcubic graph of $G$ is DP-$5$-colorable. The same goes for Theorem~\ref{thm2}. We will show that the $k$-minmimal graph for $k=5,6$ with the assumption about $\rm{mad}(G)$ is acturally not exist.

\subsection{Reducible Configurations}\label{reducible}
To show that configuration $R$ is $k$-reducible in $k$-minimal graph $G$, it suffices to show any DP-$k$-coloring of the square of $G-R$ can be extended to $G^2$. By definition, which means their always exist an independent set $I$ with $|I|=|V|$, for each list assignment $L$ with $|L(u)|=k,u\in V$ and matching assignment $M_L$. That is contradict the minimality of $G$. 
Now we present some $k$-reducible configurations for $k=5,6$.

\begin{lemma}
	$d_G(v)\ge2,v\in V$.
\end{lemma}
\begin{proof}
	For $k$-minimal graph $G$, its square graph $G^2$ is not DP-$k$-colorable. So there is no independent set $I$ with $|I|=|V|$ under a matching assignment $M_L$. Assume that $v$ is a $1$-vertex in $G$, then $(G-v)^2=G^2-v$ is DP-$k$-colorable by minimality of $G$. Let $I'$ be the independent set with $|I'|=|V|-1$. 
	The degree of $v$ is at most $3$ in $G^2$, so there are at least two colors $c_1,c_2\in L(v)$, satifying $(v,c_i)(u,c)\notin E(M_{uv}),(u,c)\in I'$, in which $u$ is any neighbor of $v$ in $G^2$. Hence, $I'\cup \{(v,c_i)\}$ is an independent set of $G^2$ with cardinality $|V|$. Which means $G^2$ is DP-$k$-colorable, that is a contradiction.
\end{proof}
Stated differently, there are only $2$-vertices and $3$-vertices in $G$.
\begin{lemma}\label{face}
	If $F$ is a $m$-face $v_1v_2\cdots v_mv_1(m\ge3)$, and only $v_1$ is a $3$-vertex, then $F$ is $k$-reducible.
\end{lemma}
\begin{proof}
	For $k$-minimal graph $G$, let $M_L$ be a matching assignment such that the cover of $G^2$ does not contain the indepedent set of cardinality $|V|$. If configuration $F$ is contained in $G$, then $(G-F)^2$ is DP-$k$-colorable, so there must exist an independent set $|I'|$ of cardinality $|V|-m$.
	Assume that $3$-vertex $v_1$ is adjacent to $u_1$ in $G-F$, and the neighbors of $u_1$ are $v_1,u_2,u_3$ for $d_G(v)=3$(or $v_1,u_2$ for $d_G(u_1)=2$). Denote:
	$$L^*(v_i)=L(v_i)\backslash \bigcup\limits_{u_j v_i\in E[R^2,G^2-R^2]}\{c'\in L(v_i): (u_j,c)(v_i,c')\in M_{u_j,v_i}, (u_j,c)\in I'\}$$
	$v_1$ has at most $3$ neighbors in $G^2-R^2$, and $|L(v_1)|\ge5$, so $|L^*(v_1)|\ge2$. Similarly, $|L^*(v_2)|\ge4,|L^*(v_m)|\ge4,|L^*(v_i)|\ge4,i=3,4,\cdots,m-1$. We can obtain an independent set $I^*=\{(v_i,c_i),i=1,2,\cdots,m\}$ by coloring $v_1,v_2,v_m,v_3,\cdots,v_{m-1}$ in order. Now $I'\cup I^*$ is an independent set of cardinality $|V|$ in the cover of $G^2$. Therefore $G^2$ is DP-$k$-colorable, a contradiction.
\end{proof}
The next lemmata are concerning specific $5$-reducible and $6$-reducible configurations of $G$.The method of prove configuration $R$ is $k$-reducible is similar to the proof of Lemma~\ref{face}: Let $V(R)=\{v_i:i=1,2,\cdots,n_1\},V(G-R)=\{u_j:j=1,2,\cdots,n_2\}$, in which $n_1+n_2=|V|$. If $H_L$ is the cover of $G^2$ with list assignment $L$ and matching assignment $M_L$. By the minimality,$G^2-R^2$ is DP-$k$-colorable,namely there exist an independent set $I'$ of cardinality $n_2$ in the cover of $G^2-R^2$. For any $v_i\in V(R)$, denote:
$$L^*(v_i)=L(v_i)\backslash \bigcup\limits_{u_j v_i\in E[R^2,G^2-R^2]}\{c'\in L(v_i): (u_j,c)(v_i,c')\in M_{u_j,v_i}, (u_j,c)\in I'\}$$
Finally, we show the existence of independent set $I^*$ of cardinality $n_1$ in $H_L[V(R)]$, by giving an order of color $V(R)$. Now $I'\cup I^*$ is an independent set of cardinality $|V|$ in the cover of $G^2$. Therefore $G^2$ is DP-$k$-colorable.

\begin{lemma}\label{5-reducible}
	The following configurations are $5$-reducible:
	\begin{enumerate}
		\item $3$-thread;
		\item $3$-face with only two $3$-vertices, and each $3$-vertex is incident to a $1$-thread;
		\item $4$-face with only two adjacent  $3$-vertices, and each $3$-vertex is incident to a $1$-thread;
		\item two $3$-faces share one common edge, and only the endpoints of common edge are $3$-vertices;
		\item two $4$-faces share one common edge, and only the endpoints of common edge are $3$-vertices;
		\item $3$-face shares one common edge with $4$-face, and only the endpoints of common edge are $3$-vertices.
	\end{enumerate}
\end{lemma}

\begin{proof}
	\begin{enumerate}
		\item For $3$-thread $v_1v_2v_3$. If $v_1$ and $v_3$ have a common neighbor, then $|L^*(v_1)|\ge3,|L^*(v_2)|\ge4,|L^*(v_3)|\ge3$; otherwise $|L^*(v_1)|\ge2,|L^*(v_2)|\ge3,|L^*(v_3)|\ge2$. In both cases, we can extend the DP-$5$-coloring of $G^2-R^2$ to $G^2$ by greedily color $v_1,v_3,v_2$ in order.
		\item For $3$-face $v_1v_2v_3v_1$, in which $3$-vertex $v_i$ is adjacent to a $2$-vertex $v'_i$ for $i=1,2$. Then $|L^*(v_3)|=5,|L^*(v_1)|\ge4,|L^*(v_2)|\ge4,|L^*(v'_1)|\ge2,|L^*(v'_2)|\ge2$. We can extend the DP-$5$-coloring of $G^2-R^2$ to $G^2$ by greedily color $v'_1,v'_2,v_1,v_2,v_3$ in order.
		\item For $4$-face $v_1v_2v_3v_4v_1$, in which $3$-vertex $v_i$ is adjacent to a $2$-vertex $v'_i$ for $i=1,2$. Then $|L^*(v_3)|=|L^*(v_4 )|=5,|L^*(v_1)|\ge4,|L^*(v_2)|\ge4, |L^*(v'_1)|\ge2,|L^*(v'_2)|\ge2$. We can extend the DP-$5$-coloring of $G^2-R^2$ to $G^2$ by greedily color $v'_2,v'_1,v_1,v_2,v_3,v_4$ in order.
		\item For configuration contains two adjacent $3$-faces $uvw_1u$ and $uvw_2u$. Its square graph is isomorphism to complete graph $K_4$, which is obiviously DP-$5$-colorable.
		\item For configuration sonsist of two adjacent $4$-faces $v_1v_2v_3v_4v_1$ and $v_1v_2v_5v_6v_1$, in which only $v_1,v_2$ are $3$-vertices. Then $|L^*(v_i)|=5,i=1,2,\cdots,6$. We can extend the DP-$5$-coloring of $G^2-R^2$ to $G^2$ by greedily color $v_1,v_2,\cdots,v_6$ in order.
		\item Let $R$ be the configuration with one $3$-face $v_1v_2v_3v_1$ adjacent to a $4$-face $v_1v_2v_4v_5v_1$, in which only $v_1,v_2$ are $3$-vertices. Then its square graph is isomorphic to complete graph $K_5$, which is obiviously DP-$5$-colorable.
	\end{enumerate}
\end{proof}
\begin{lemma}\label{6-reducible}
	The following configurations are $6$-reducible:
	\begin{enumerate}
		\item $2$-thread;
		\item $4$-face with only two nonadjacent $3$-vertices;
		\item $F_{2,3}$ depicted in Figure~\ref{fig:6-reducible}.
	\end{enumerate}
\end{lemma}
\begin{proof}
	\begin{enumerate}
		\item For $2$-thread $v_1v_2$, if $v_1,v_2$ have a common neighbor, then $|L^*(v_1)|\ge4,|L^*(v_2)|\ge4$; otherwise, $|L^*(v_1)|\ge2,|L^*(v_2)|\ge2$. We can extend the DP-$6$-coloring of $G^2-R^2$ to $G^2$ by greedily color $v_1,v_2$ in order.
		\item For $4$-face $v_1v_2v_3v_4v_1$ with only two $3$-vertices $v_1,v_3$. Then in $(G-\{v_1,v_2,v_3,v_4\})^2$, $|L^*(v_2)|\ge4,|L^*(v_4 )|\ge4,|L^*(v_1)|\ge3,|L^*(v_3)|\ge3$. We can greedily color $v_1,v_2$ in order to obtain a DP-$6$-coloring.
		\item Let $F_{2,3}$ dispected in Figure~\ref{fig:6-reducible} be the configuration consists of three $2$-vertices $u_1,u_2,u_3$ and their two common neighbors $v_1,v_2$ with $d_G(v_1)=d_G(v_2)=3$. Then its square graph is isomorphic to complete graph $K_5$, which is obiviously DP-$6$-colorable. Thus $F_{2,3}$ is a $6$-reducible configuration.
	\end{enumerate}
\end{proof}
\begin{figure}[h]
	\centering
	\includegraphics[width=4cm]{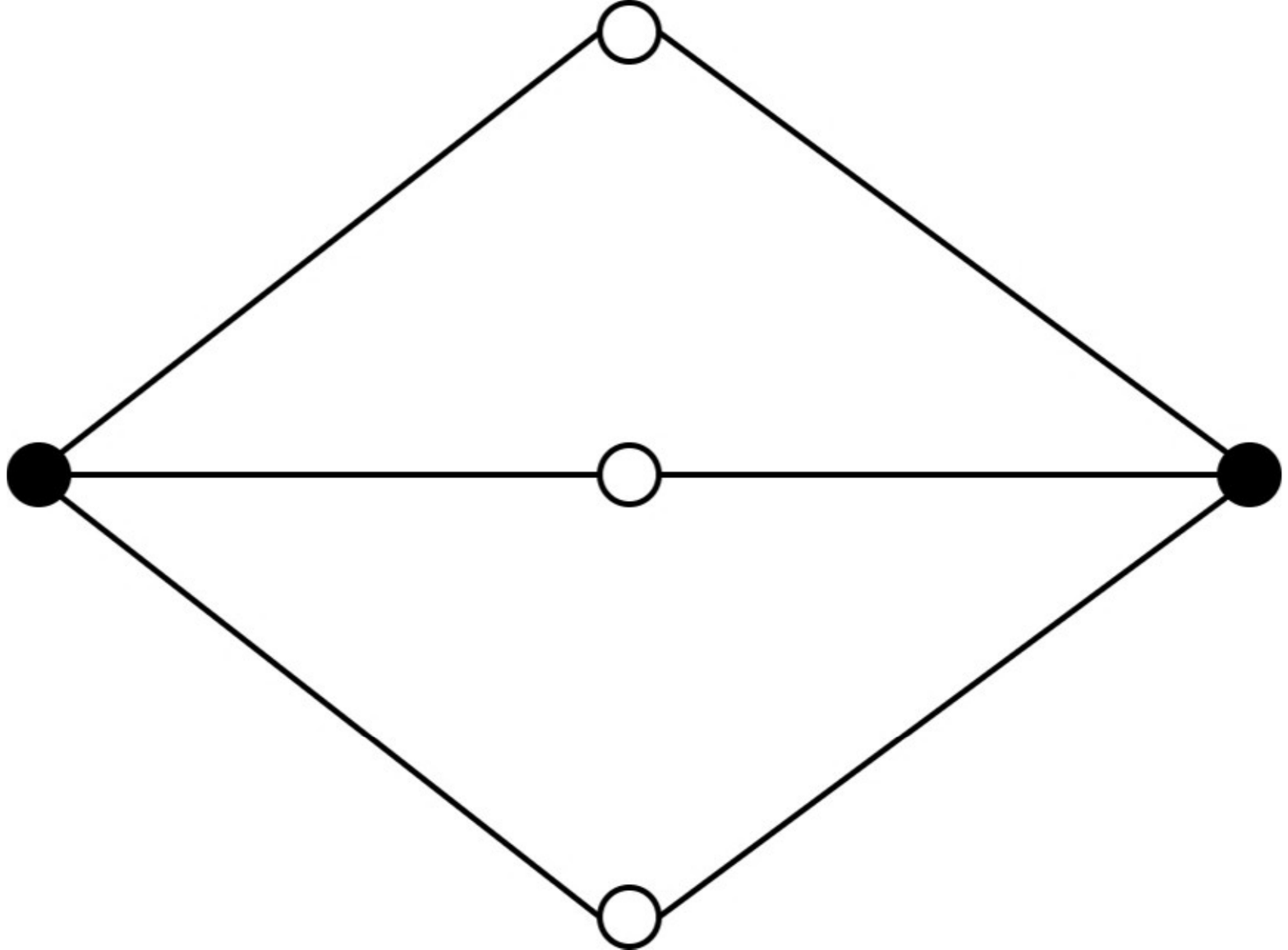}
	\caption{The graph $F_{2,3}$}
	\label{fig:6-reducible}
\end{figure}
Furtherly, we descript the $k$-minimal graph for $k=5,6$.
\begin{lemma}\label{5-minimal}
	For $5$-minimal graph $G$ and $3$-vertex $v\in V$, we have:
	\begin{enumerate}
		\item if $v$ is not adjacent to any $3$-vertex, then $G_2(v)$ is isomorphic to the subgraph of $Y_{2,2,2}$;
		\item if $v$ is adjacent to a $3$-vertex $u$, then $G_2(v),G_2(u)$ are both isomorphic to the subgraph of $Y_{0,2,2}$.
	\end{enumerate}
\end{lemma}
\begin{proof}
	\begin{enumerate}
		\item If $3$-vertex $v\in V$ only have neighbors of degree $2$, and $3$-vertices $v_i(i=1,2,3)$ are connected to $v$ by $l_i$-thread($i=1,2,3$) respectively. By Lemma~\ref{face}, the $m$-face with only one $3$-vertex is reducible, so $v_1,v_2,v_3$ can not coincide with $v$. Without loss of generality, we may assume that $l_1\le l_2\le l_3$. Then $l_3\le2$,by Lemma~\ref{5-reducible}(1). Thus, $G_2(v)$ is isomorphic to the subgraph of $Y_{2,2,2}$.
		\item If $3$-vertex $v\in V$ only have one neighbor of degree $3$,say $u$. And assume $3$-vertices $v_i(i=1,2)$ are connected to $v$ by $l_i$-thread($i=1,2$); $3$-vertices $u_i(i=1,2)$ are connected to $u$ by $l'_i$-thread($i=1,2)$. Without loss of generality, we may assume $l_1\le l_2\le l'_2,l'_1\le l'_2$.\\
		\textbf{Case 1}: none of $u_1,u_2$ coincide with $v$, and none of $v_1,v_2$ coincide with $u$. By Lemma~\ref{5-reducible}(1), $l'_2\le2$. Thus $G_2(v),G_2(u)$ are both isomorphic to the subgraph of $Y_{0,2,2}$;\\
		\textbf{Case 2}: vertex $u_1$ is coincide with $v$, while $u_2$ is not coincide with $v$. Then $u,v$(or $u_1$),$u_2,v_1,v_2$ and $l_i$-thread,$l'_i$-thread($i=1,2$) may form the induced subgraph isomorphic to the configurations in Figure~\ref{fig:5-reducible}(a). Which are both $5$-reducible by Lemma~\ref{5-reducible}(2)(3).\\
		\textbf{Case 3}: both $u_1$ and $u_2$ are coincide with $v$. Then $u$(or $v_1,v_2$),$v$(or $u_1,u_2$) and $l_i$-thread,$l'_i$-thread($i=1,2$) may form the induced subgraph isomorphic to the configurations in Figure~\ref{fig:5-reducible}(b). Which are both $5$-reducible by Lemma~\ref{5-reducible}(4)(5)(6).\\
		In conclusion, $G_2(v),G_2(u)$ are both isomorphic to the subgraph of $Y_{0,2,2}$.
	\end{enumerate}
\end{proof}
\begin{figure}[h]
	\centering
	\includegraphics[width=0.8\textwidth]{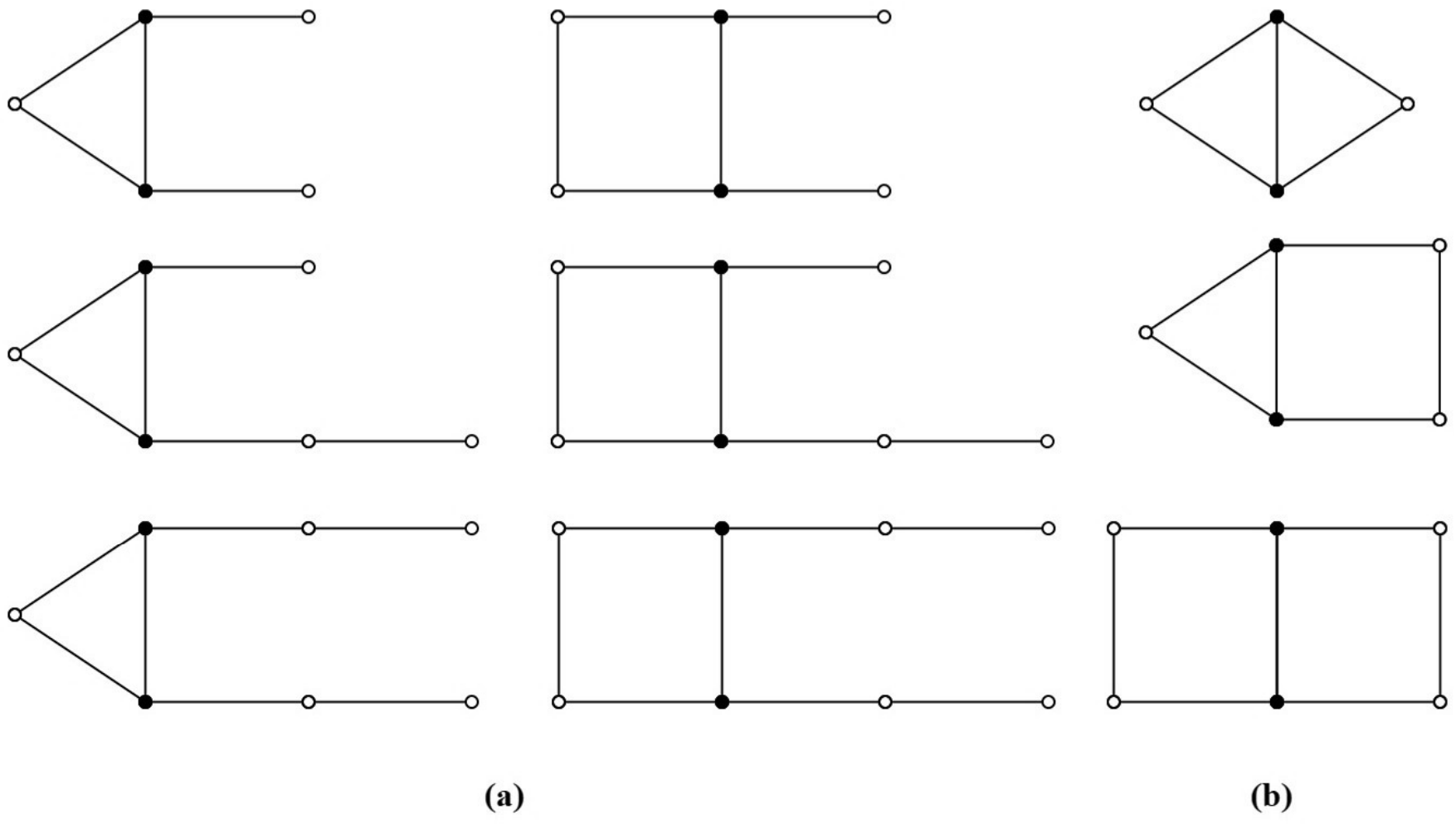}
	\caption{The $5$-reducible configurations}
	\label{fig:5-reducible}
\end{figure}

\begin{lemma}\label{6-minimal}
	For $6$-minimal graph $G$ and $3$-vertex $v$, $G_2(v)$ is isomorphic to $Y_{1,1,1}$ or subgraph of $Y_{0,1,1}$.
\end{lemma}
\begin{proof}
	Let $v_i(i=1,2,3)$ be the $3$-vertices connected to $v$ by $l_i$-thread($i=1,2,3$) respectively, and assume $l_1\le l_2\le l_3$. By Lemma~\ref{6-reducible}(1), $l_3\le 1$. For $l_1=0$, if $v_2,v_3$ are coincide, then $G_2(v)$ will form a $4$-face with two nonadjacent $3$-vertices. Which is $6$-reducible by Lemma~\ref{6-reducible}(2), so $G_2(v)$ is isomorphic to subgraph of $Y_{0,1,1}$. If $l_1=1$, $v_1,v_2,v_3$ can not coincide by Lemma~\ref{face} and Lemma~\ref{6-reducible}(2)(3), thus $G_2(v)$ is isomorphic to $Y_{1,1,1}$.
\end{proof}

\subsection{Main Results}\label{main}
Now we complete the proof by discharging method. Let the initial charge of $k$-vertex be its degree $k$, and let $ch^*(v)$ be the final charge of $v$. 
\begin{proof}[Proof of Theorem~\ref{thm1}]
	For $5$-minimal graph $G$, we only need one discharging rule:
	\begin{description}
		\item[R1] Each $3$-vertex gives $\frac{1}{4}$ to each adjacent $2$-vertex.
	\end{description}
	We show that after the discharging proceduce, the final charge of any vertex is at least $\frac{9}{4}$, which is in contradiction to the assumption of the maximum average degree of $5$-minimal graph $G$.\par
	For any $3$-vertex $v$, it has at most three neighbors of degree $2$ by Lemma~\ref{5-minimal}. Thus, $ch^*(v)\ge3-3\cdot\frac{1}{4}=\frac{9}{4}$. 
	For any $2$-vertex $v$, it must on a $1$-thread or a $2$-thread. In the former case, two neighbors of $v$ are both $2$-vertices, then $ch^*(v)=2+2\cdot\frac{1}{4}=\frac{10}{4}=\frac{5}{2}$; in the later case, $v$ is only adjacent to one $3$-vertex, thus $ch^*(v)=2+\frac{1}{4}=\frac{9}{4}$.
\end{proof}
	In fact, the bound we give above is tight. Consider the graph $F_{2,6}$ depicted in Figure~\ref{graph-tight}, which has average degree $\frac{9}{4}$. But the square of the subgraph induced by $\{u_1,u_2,v_1,v_2,w_1,w_2\}$ in $(F_{2,6})^2$ is isomorphic to complete graph $K_6$, which is not DP-$5$-colorable.
\begin{figure}[h]
	\centering
	\includegraphics[width=5cm]{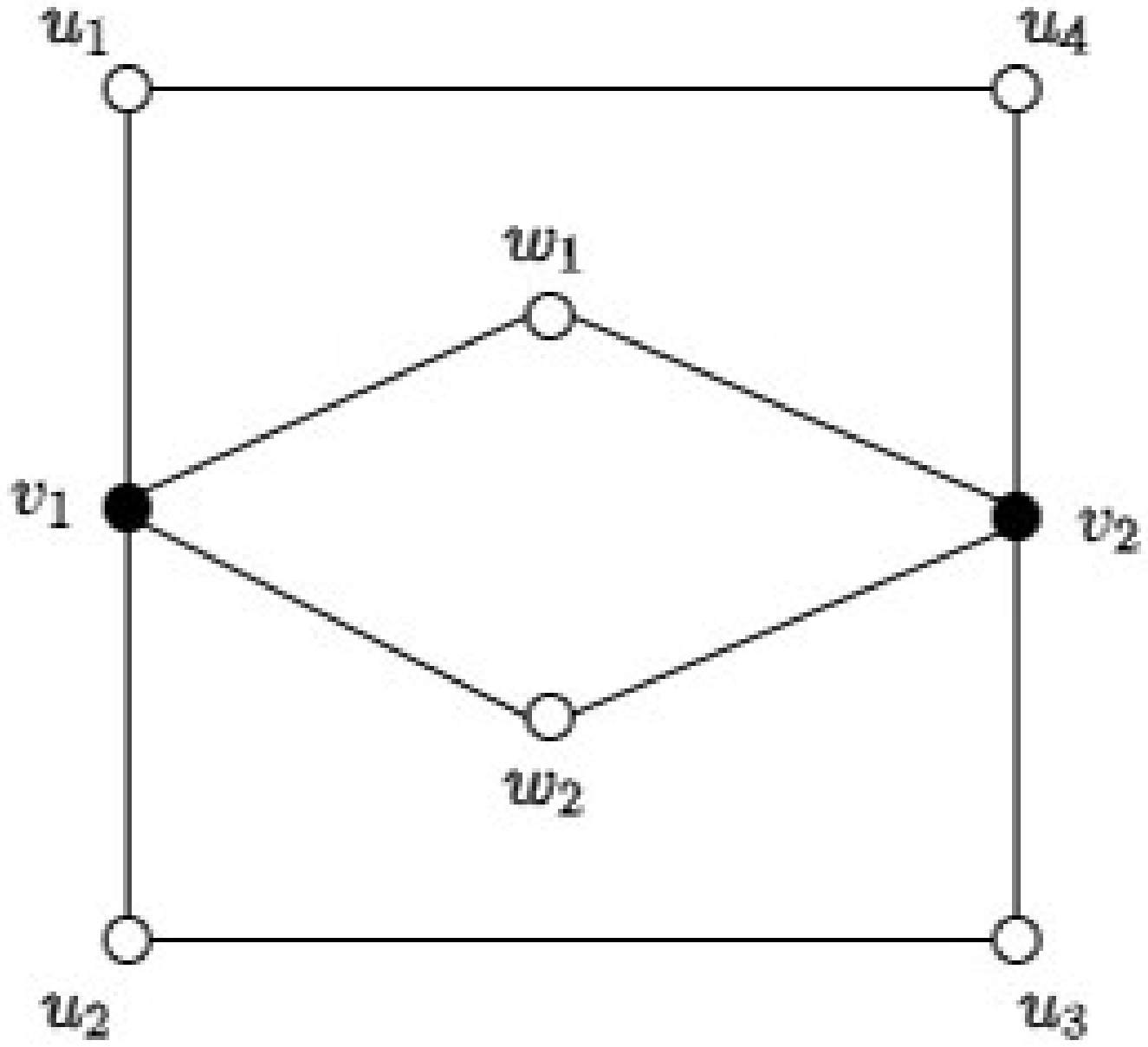}
	\caption{The graph $F_{2,6}$}
	\label{graph-tight}
\end{figure}

\begin{proof}[Proof of Theorem~\ref{thm2}]
	For $6$-minimal graph $G$, the only discharging rule is as follows:
	\begin{description}
		\item[R2] Each $3$-vertex gives $\frac{1}{5}$ to each adjacent $2$-vertex.
	\end{description}
	We show that if we redistribute charges by the rule above, the final charge of any vertex is at least $\frac{12}{5}$, which also leads to a contradiction.\par
	For any $3$-vertex $v$, it has at most three neighbors of degree $2$ by Lemma~\ref{6-minimal}. Then we obtain $ch^*(v)\ge 3-3\cdot\frac{1}{5}=\frac{12}{5}$. 
	For any $2$-vertex $v$, it only adjacent to $3$-vertices by Lemma~\ref{6-reducible}(1). Thus $ch^*(v)=2+2\cdot\frac{1}{5}=\frac{12}{5}$.
\end{proof}
Unfortunately, we don't know whether this bound is tight or not.

\end{document}